\documentclass[10pt]{article}

\DeclareMathAlphabet{\mathpzc}{OT1}{pzc}{m}{it}

\usepackage{amsmath}
\usepackage{amssymb}
\usepackage{amsfonts}
\usepackage{amsthm}
\usepackage{mathrsfs}
\usepackage{enumerate}
\usepackage[pdftex]{color, graphicx}

\newcommand{\cC}{{\mathcal C}}

\newcommand{\pa}{\parallel}

\newcommand{\C}{\mathcal{C}}

\newcommand{\N}{\mathcal{N}}

\newcommand{\I}{\mathcal{I}}

\allowdisplaybreaks

\newtheorem{thm}{Theorem}[section]
\newtheorem{lem}[thm]{Lemma}

\newtheorem{prop}[thm]{Proposition}
\newtheorem{cor}[thm]{Corollary}

\begin{document}

\renewcommand{\thefootnote}{\arabic{footnote}}
 	
\title{{\bf A cevian locus and the geometric construction of a special elliptic curve}}

\author{\renewcommand{\thefootnote}{\arabic{footnote}}
Igor Minevich and Patrick Morton}
\maketitle

\begin{section}{Introduction}

In previous papers \cite{mm1}, \cite{mm3}, \cite{mmv} we have studied several conics defined for an ordinary triangle $ABC$ relative to a given point $P$, not on the sides of $ABC$ or its anticomplementary triangle, including the inconic $\mathcal{I}$ and circumconic $\tilde \cC_O$.  These two conics are defined as follows.  Let $DEF$ be the cevian triangle of $P$ with respect to $ABC$ (i.e., the diagonal triangle of the quadrangle $ABCP$).  Let $K$ denote the complement map and $\iota$ the isotomic map for the triangle $ABC$, and set $P'=\iota(P)$ and $Q=K(P')=K(\iota(P))$.  Furthermore, let $T_P$ be the unique affine map taking $ABC$ to $DEF$, and $T_{P'}$ the unique affine map taking $ABC$ to the cevian triangle for $P'$.  \medskip

The inconic $\mathcal{I}$ for $P$ with respect to $ABC$ is the unique conic which is tangent to the sides of $ABC$ at the traces (diagonal points) $D,E,F$.  (See \cite{mm1}, Theorem 3.9 for a proof that this conic exists.) If $\N_{P'}$ is the nine-point conic of the quadrangle $ABCP'$ relative to the line at infinity $l_\infty$, then the circumconic $\tilde \cC_O$ is defined to be $\tilde \cC_O=T_{P'}^{-1}(\N_{P'})$.  (See \cite{mm3}, Theorems 2.2 and 2.4; and \cite{co1}, p. 84.)  These two conics are generalizations of the classical incircle and circumcircle of a triangle, and coincide with these circles when the point $P=Ge$ is the Gergonne point of the triangle.  In that case, the point $Q=I$ is the incenter.  In general, the point $Q$ is the center of the inconic $\I$.  The center $O$ of the circumconic $\tilde \cC_O$ is given by the affine formula
$$O=T_{P'}^{-1} \circ K(Q),$$
since the center of the conic $\N_{P'}$ turns out to be $K(Q)$.  \medskip

\begin{figure}
\[\includegraphics[width=5.5in]{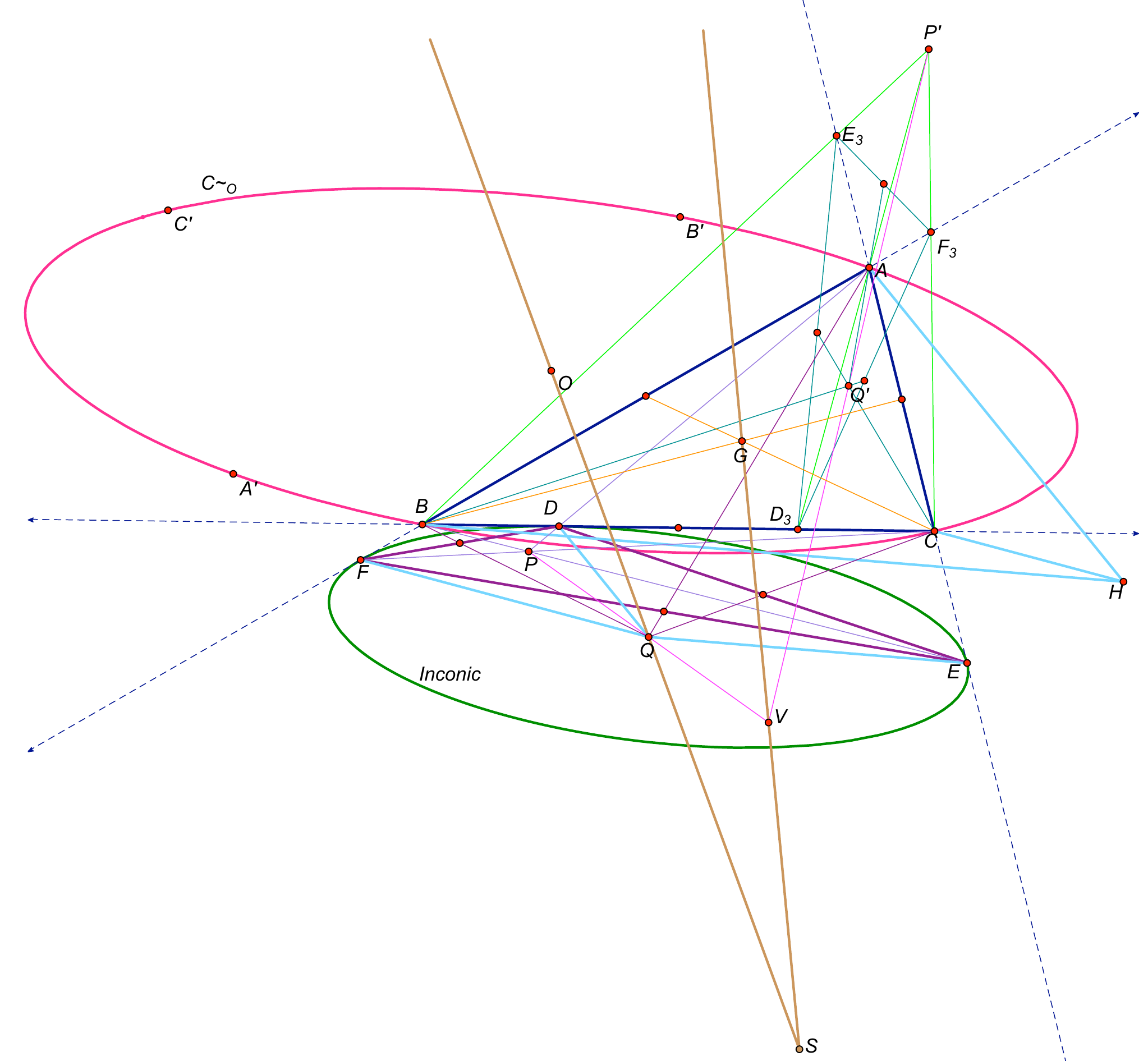}\]
\caption{The conics $\tilde{\mathcal{C}}_O$ (strawberry) and $\mathcal{I}$ (green).}
\label{fig:conics}
\end{figure}

We also showed in \cite{mm3}, Theorem 3.4, that the affine map $\textsf{M}=T_P \circ K^{-1} \circ T_{P'}$ is a homothety or translation which maps the circumconic $\tilde{\mathcal{C}}_O$ to the inconic $\mathcal{I}$.  If $G$ is the centroid of $ABC$ and $Q'=K(P)$, then the center of the map $\textsf{M}$ is the point
$$S=OQ \cdot GV = OQ \cdot O'Q', \ \ \textrm{where} \ V=PQ \cdot P'Q',$$
and $O'=T_{P}^{-1} \circ K(Q')$ is the generalized circumcenter for the point $P'$. \medskip

In \cite{mmv} we showed that for a fixed triangle $ABC$ the locus of points $P$, for which the map $\textsf{M}$ is a translation, is an elliptic curve minus $6$ points, and that this elliptic curve has infinitely many points defined over $\mathbb{Q}(\sqrt{2})$.  Thus, there are infinitely many points $P$ for which the conics $\I$ and $\tilde \cC_O$ are congruent to each other.  In this note we determine the remaining points $P$ for which these two conics are congruent by determining (synthetically) the locus of points $P$ for which the map $\textsf{M}$ is a half-turn.  We show, for example, that $\textsf{M}$ being a half-turn is equivalent to the point $P$ lying on the circumconic $\tilde \cC_{O'}$, where $O'=T_P^{-1} \circ K(Q')$ is the generalized circumcenter for the point $P'$, and this is also equivalent to the point $P'=\iota(P)$ lying on the circumconic $\tilde \cC_O$.  (By contrast, we showed in \cite{mmv} that $\textsf{M}$ is a translation if and only if the point $P$ lies on $\tilde \cC_O$.)  This is interesting, since if $P=Ge$ is the Gergonne point of $ABC$, then $P'=Na$ is the Nagel point, which always lies {\it inside} the circumcircle $\tilde \cC_O$.  Given triangle $ABC$, the locus of all such points $P$ turns out to be another elliptic curve (minus $6$ points; see Theorem \ref{thm:curveE}).  As in \cite{mmv}, this elliptic curve can be constructed synthetically using a locus of affine maps defined for points on certain open arcs of a conic.  In \cite{mmv} the latter conic was a hyperbola, while here the conic needed to construct the elliptic curve is a circle. (See Figure \ref{fig:3.3}.) \medskip

We adhere to the notation of \cite{mm1}: $D_0E_0F_0$ is the medial triangle of $ABC$, with $D_0$ on $BC$, $E_0$ on $CA$, $F_0$ on $AB$ (and the same for further points $D_i,E_i,F_i$); $DEF$ is the cevian triangle associated to $P$; $D_2E_2F_2$ the cevian triangle for $Q=K \circ \iota(P)$; $D_3E_3F_3$ the cevian triangle for $P'=\iota(P)$; and $G$ the centroid of $ABC$.  As above, $T_P$ and $T_{P'}$ are the unique affine maps taking triangle $ABC$ to $DEF$ and $D_3E_3F_3$, respectively.  See \cite{mm1} and \cite{mm2} for the properties of these maps.  Also, the generalized orthocenter for $P$ with respect to $ABC$ is the point $H=K^{-1}(O)$, which is also the intersection of the lines through the vertices $A,B,C$ which are parallel, respectively, to the lines $QD, QE, QF$.  Finally, the point $Z$ is defined to be the center of the cevian conic $\cC_P=ABCPQ$.  (See \cite{mm2}.)  \smallskip

We also refer to the papers \cite{mm1}, \cite{mm2}, \cite{mm3}, and \cite{mm4} as I, II, III, IV, respectively.  See \cite{ac}, \cite{co1}, \cite{co2} for results and definitions in triangle geometry and projective geometry.

\end{section}

\begin{section}{The locus of $P$ for which $\textsf{M}$ is a half-turn.}

In this section we determine necessary and sufficient conditions for the map $\textsf{M}$ to be a half-turn.  We start with the following lemma.

\begin{lem}
\label{lem:median}
\begin{enumerate}[(a)]
\item If the point $P$ (not on a side of $ABC$ or $K^{-1}(ABC)$) lies on the Steiner circumellipse $\iota(l_\infty)$ of $ABC$, then the map $\textsf{M}= T_P \circ K^{-1} \circ T_{P'}$ is a homothety with ratio $k=4$, and is therefore not a half-turn.
\item If the point $P$ lies on a median of triangle $ABC$, but does not lie on the Steiner circumeellipse $\iota(l_\infty)$ of $ABC$, then $\textsf{M}= T_P \circ K^{-1} \circ T_{P'}$ is not a half-turn.
\end{enumerate}
\end{lem}

\begin{proof}
To prove (a), we use the result of I, Theorem 3.14, according to which $P$ lies on $\iota(l_\infty)$ if and only if the maps $T_P$ and $T_{P'}$ satisfy $T_P T_{P'}=K^{-1}$.  If this condition holds, then because the map $\textsf{M}$ is symmetric in $P$ and $P'$ (see III, Proposition 3.12b or \cite{mmv}, Section 3), 
$$\textsf{M}=T_{P'}K^{-1}T_P=T_{P'}T_P T_{P'}T_P=(T_{P'}T_P)^2=(T_P^{-1}K^{-1}T_P)^2=T_P^{-1}K^{-2}T_P.$$
The similarity ratio of the dilatation $K^{-1}$ is $-2$, so the similarity ratio of $K^{-2}$ is $4$, which proves part (a). \smallskip

For (b), suppose $P$ lies on the median $AG$ ($G$ the centroid of $ABC$) and the map $\textsf{M}$ is a half-turn. Then, since $P'$ also lies on $AG$, we have that $D=D_0=D_3$, so that
$$\textsf{M}(A)=T_P K^{-1} T_{P'}(A) = T_P K^{-1}(D_3)=T_P K^{-1}(D_0)=T_P(A)=D=D_0$$
 and the center $S$ of $\textsf{M}$ is the midpoint of $AD_0$.  In particular, $\textsf{M}(B)$ and $\textsf{M}(C)$ are the reflections in $S$ of $B$ and $C$ on the line $\ell=K^{-1}(BC)$.  We claim that the line $\ell$ is tangent to the circumconic $\tilde \cC_O$ at the point $A$.  This is because the affine reflection $\rho$ through the line $AG=AP$ in the direction of the line $BC$ takes the triangle $ABC$ to itself, and maps $P$ to $P$, so it also takes the circumconic $\tilde \cC_O$ to itself.  Hence the tangent to $\tilde \cC_O$ at $A$ maps to itself, which implies that it must be parallel to $BC$ (since the only other ordinary fixed line is $AG$, which lies on the center $O$ of $\tilde \cC_O$ and cannot be a tangent at an ordinary point).  But $\ell$ is the unique line through $A$ parallel to $BC$, so $\ell$ must be the tangent.  (Also see III, Corollary 3.5.)  It follows that the points $\textsf{M}(B)$ and $\textsf{M}(C)$, neither of which is $A$, must be exterior points of the conic $\tilde \cC_O$.  On the other hand, we claim that $D=D_0$ is an {\it interior} point of $\tilde \cC_O$.  This is because the segment $BC$, parallel to the tangent $\ell$ at $A$, is a chord of $\tilde \cC_O$, and $B$ and $C$ lie on the same branch of $\tilde \cC_O$, if the latter is a hyperbola (any tangent to a hyperbola separates the two branches).  It follows that the segments $D\textsf{M}(B)$ and $D\textsf{M}(C)$ join the interior point $D$ to exterior points, and so must each contain a point on $\tilde \cC_O$.  However, $\textsf{M}$ is a half-turn mapping the circumconic $\tilde \cC_O$ to the inconic $\I$, so that $\textsf{M}(\I)=\tilde \cC_O$.  Hence, $\tilde \cC_O$ must be inscribed in the triangle $\textsf{M}(ABC) = D\textsf{M}(B)\textsf{M}(C)$, meaning that $\tilde \cC_O$ touches all three extended sides of the triangle.  But by what we just showed the intersections of $\tilde \cC_O$ with the sides of $D\textsf{M}(B)\textsf{M}(C)$ lie on the segments joining the vertices.  Hence, the point $D$ lies on the two tangents $b=D\textsf{M}(B)$ and $c=D\textsf{M}(C)$ to $\tilde \cC_O$, implying that $D$ is an {\it exterior} point of $\tilde \cC_O$.  This contradiction proves the lemma.
\end{proof}

\begin{prop}
\label{prop:half-turn}
If the points $P$ and $P'$ are ordinary and do not lie on the sides or medians of triangles $ABC$ and $K^{-1}(ABC)$, and $H$ does not coincide with a vertex of $ABC$, the following are equivalent:
\begin{enumerate}[(1)]
\item $\textsf{M} = T_P \circ K^{-1} \circ T_{P'}$ is a half-turn;
\item $P$ is on $\tilde{\cC}_{O'}$;
\item $P'$ is on $\tilde{\cC}_{O}$;
\item $T_P(P) = O'$;
\item $T_{P'}(P') = O$;
\item $O'$ lies on $\N_P$;
\item $O$ lies on $\N_{P'}$.
\end{enumerate}
\end{prop}

\begin{proof} 
First note that (4) $\iff$ (5): this follows on applying the affine reflection $\eta$ from Part II to (4) ($\eta$ is the harmonic homology with axis $GZ$ and center $PP' \cdot l_\infty$), and using that
$$\eta(P) = P', \ \ \eta(O) = O', \ \ \textrm{and} \ \ \eta \circ T_P = T_{P'} \circ \eta.$$
By III, Proposition 3.12, $\textsf{M}$ commutes with $\eta$, since $\eta$ commutes with $K$ and
$$\textsf{M} \circ K^{-1} = (T_P \circ K^{-1}) \circ (T_{P'} \circ K^{-1})  = (T_{P'} \circ K^{-1}) \circ (T_{P} \circ K^{-1}).$$
Hence, $T_{P'} \circ K^{-1} \circ T_P = \eta \textsf{M} \eta = \textsf{M}$.  This shows that the locus of points $P$ for which $\textsf{M}$ is a half-turn is invariant under $P \rightarrow P'$.    \smallskip
 
We now show that (1) is equivalent to (4) and (5).  Namely, if \textsf{M} is a half-turn, then since $\textsf{M}(O')=Q'$, we have to have $\textsf{M}(Q')=O'$.  But
$$\textsf{M}(Q') = T_P K^{-1} T_{P'}(Q') = T_P K^{-1}(Q') =T_P(P),$$
and so $O' = T_P(P)$.  Conversely, if $O' = T_P(P)$, then $\textsf{M}(Q')= O'$, so $\textsf{M}^2(O') = O'$, which implies that \textsf{M} must be a homothety with similarity ratio $k = \pm 1$, since $\textsf{M}^2$ fixes the point $O' \neq S$.  However, $k$ can't be $+1$, since in that case \textsf{M} is the identity and $O' = Q'$, impossible by the argument of III, Theorem 3.9.  Therefore, $k = -1$, so \textsf{M} is a half-turn.  (Note that $O' \neq S$, since otherwise $O=\eta(O')=\eta(S)=S$; but the points $O, O', Q, Q'$ are distinct, by the proof of III, Theorem 3.9, as long as $P$ does not lie on a median of $ABC$ or on $\iota(l_\infty)$.)  This shows that (1) $\iff$ (4) $\iff$ (5). \smallskip

Furthermore, if (3) holds, then $P'$ is on $\tilde{\cC}_O$, so the latter conic lies on the vertices of quadrangle $ABCP'$ (since $\tilde{\cC}_O$ is a circumconic), so the center $O$ must lie on $\N_{P'}$, by definition of the 9-point conic.  (See Part III, paragraph before Prop. 2.4.) Thus (3) implies (7).  Also, (7) implies (3), because $O$ being on $\N_{P'}$ implies $O$ is the center of a conic on $ABCP'$.  If $O$ is not the midpoint of a side of $ABC$ (which holds if and only if $H$ is not a vertex), there is a unique such conic, namely $\tilde{\cC}_O$.  Hence $P'$ lies on this conic.  This shows that (3) $\iff$ (7).  Similarly, (2) $\iff$ (6). \smallskip

Now suppose that (3) holds.  Then (7) holds, so $P' \in \tilde{\cC}_O$.  But $P' \in \cC_P$, so $P'$ is the fourth intersection of the circumconics $\cC_P$ and $\tilde{\cC}_O$, and therefore coincides with the point $\tilde Z = R_OK^{-1}(Z)$, by III, Theorem 3.14; here $R_O$ is the half-turn about $O$ and $Z$ is the center of $\cC_P$.
Now, in the proof of III, Theorem 3.14 we showed that $T_{P'}(\tilde{Z})=T_{P'}(P')$ lies on $OP'$. But $T_{P'}(P')$ lies on $OQ$ (III, Proposition 3.12), so this forces $T_{P'}(P') = O$, i.e. (5), provided we can show that the line $OP'$ is distinct from $OQ$. \smallskip

However, if $OP' = OQ$, then $OP' = P'Q = QG$ so $O, Q, G$ are collinear. Then $K^{-1}(O) = H$ is also on this line, so $Q, H, P'$ are all on this line.  We claim that these three points, $Q, H, P'$, must all be distinct by our hypothesis on $P$.  If $P' = Q =K(P')$ then $P' = G = P$, which can't hold because $P$ is not on a median of $ABC$.  If $Q = H$, then using the map $\lambda=T_{P'} \circ T_P^{-1}$ and III, Theorem 2.7 gives  $Q=\lambda(H)=\lambda(Q) = P'$, so $P'=G$.  Finally, if $P' = H$, then taking complements gives that $Q = O =T_{P'}^{-1}(K(Q))$ so $T_{P'}(Q) = K(Q)$, implying (by I, Theorem 3.7) that $P' = K(Q)=K^2(P')$ and therefore $P' = G = P$ once again. Therefore, the three distinct points $Q, H, P'$, which all lie on the conic $\cC_P$, are collinear, which is impossible. This shows that $OP' \cap OQ = O$, so $T_{P'}(P')=O$ and (3) $\Rightarrow$ (5) $\Rightarrow$ (1). \smallskip

For the rest, it suffices to show that (1) $\Rightarrow$ (3).  This is because of the symmetry of $\textsf{M}$ in $P$ and $P'$: for example, (3) $\Rightarrow$ (1) $\Rightarrow$ (2) (switching $P$ and $P'$) and conversely, so (2) and (3) are equivalent, as are (6) and (7), and everything is equivalent to (1).  Now assume (1).  We will prove (7).  Since (1) implies (5), we know $T_{P'}(P') = O$, so $T_{P'}(OP') = K(Q)O$ by the formula for $O$.  But by III, Corollary 3.13(b) we know $K^{-1}(Z)$ lies on $OP'$, so $T_{P'} \circ K^{-1}(Z) = Z$ (III, Prop. 3.10) lies on $K(Q)O$. But $Z$ also lies on $QN=K(P'O)$, which is parallel to $OP'$.  This easily implies $Z$ is an ordinary point and $QZ \pa OP'$; for this, note $Q \neq Z$ since $Z$ is the center of the conic $\cC_P$, while $Q$ is an ordinary point lying {\it on} $\cC_P$.  Also, $Z \neq K(Q)$, since $T_P \circ K^{-1}(Z) = Z$, while $T_P \circ K^{-1}(K(Q)) = Q$.  Therefore, using the fact that the lines $OP'$ and $OQ$ are distinct, it follows that $QK(Q)Z$ and $P'K(Q)O$ are similar triangles.  But $K(Q)$ is the midpoint of segment $P'Q$, so these triangles must be congruent.  Hence, $K(Q)$ is also the midpoint of segment $OZ$, and $O$ is the reflection of $Z$ in the point $K(Q)$.  Since $K(Q)$ is the center of $\N_{P'}$ and $Z$ lies on $\N_{P'}$, this means $O$ also lies on $\N_{P'}$, which is (7).
\end{proof}

\begin{cor}
\label{cor:ZS}
With the hypotheses of Proposition \ref{prop:half-turn}, the map $\textsf{M}$ is a half-turn if and only if $K^{-1}(S) = Z$, which holds if and only if $QZP'O$ is a parallelogram.
\end{cor}

\begin{proof}
If $\textsf{M}$ is a half-turn, the above argument shows that segments $QZ \cong OP'$; hence, $QZP'O$ is a parallelogram, and $P'Z \pa QO$.  Since $K(P')=Q$, this implies that line $P'Z$ is the same as the line $K^{-1}(QO) = P'H$ and $Z$ is the midpoint of $P'H$, since the map $K^{-1}$ doubles lengths of segments. But $S$ is the midpoint of $QO$, so $K^{-1}(S)=Z$ is the midpoint of $P'H$.  Conversely, if $K^{-1}(S) = Z$ then $Z$ lies on $K^{-1}(OQ)=P'H$ (since $S$ lies on $OQ$) and because $QZ \pa P'O$, $QZP'O$ is a parallelogram with center $K(Q)$, the midpoint of $P'Q$.  Hence, $O$ lies on $\N_{P'}$ and $\textsf{M}$ is a half-turn.
\end{proof}

\noindent {\bf Remark.} The condition of Corollary \ref{cor:ZS} is a necessary condition for $\textsf{M}$ to be a half-turn, without the hypothesis that $H$ not be a vertex.  This follows from the last paragraph in the proof of Proposition \ref{prop:half-turn}, since that hypothesis is not used to prove that (1) $\Rightarrow$ (7).  Lemma \ref{lem:median} then shows that $K^{-1}(S) = Z$ is a necessary condition for $\textsf{M}$ to be a half-turn, without any extra hypotheses.  \medskip

\begin{prop}
\label{prop:tangents}
If the hypotheses of Proposition \ref{prop:half-turn} hold and the map $\textsf{M}$ is a half-turn, then:
\begin{enumerate}[(1)]
\item The lines $O'P$ and $OP'$ are tangents to the conic $\cC_P$ at $P$ and $P'$.
\item The point $V = PQ \cdot P'Q'$ is the midpoint of segment $OO'$ and line $OO'=K^{-1}(PP')$.
\item On the line $GV$, the signed ratios $\frac{GS}{SV}=\frac{5}{3}$ and $\frac{ZG}{GV}=\frac{5}{4}$.
\end{enumerate}
\end{prop}

\begin{figure}
\[\includegraphics[width=5.5in]{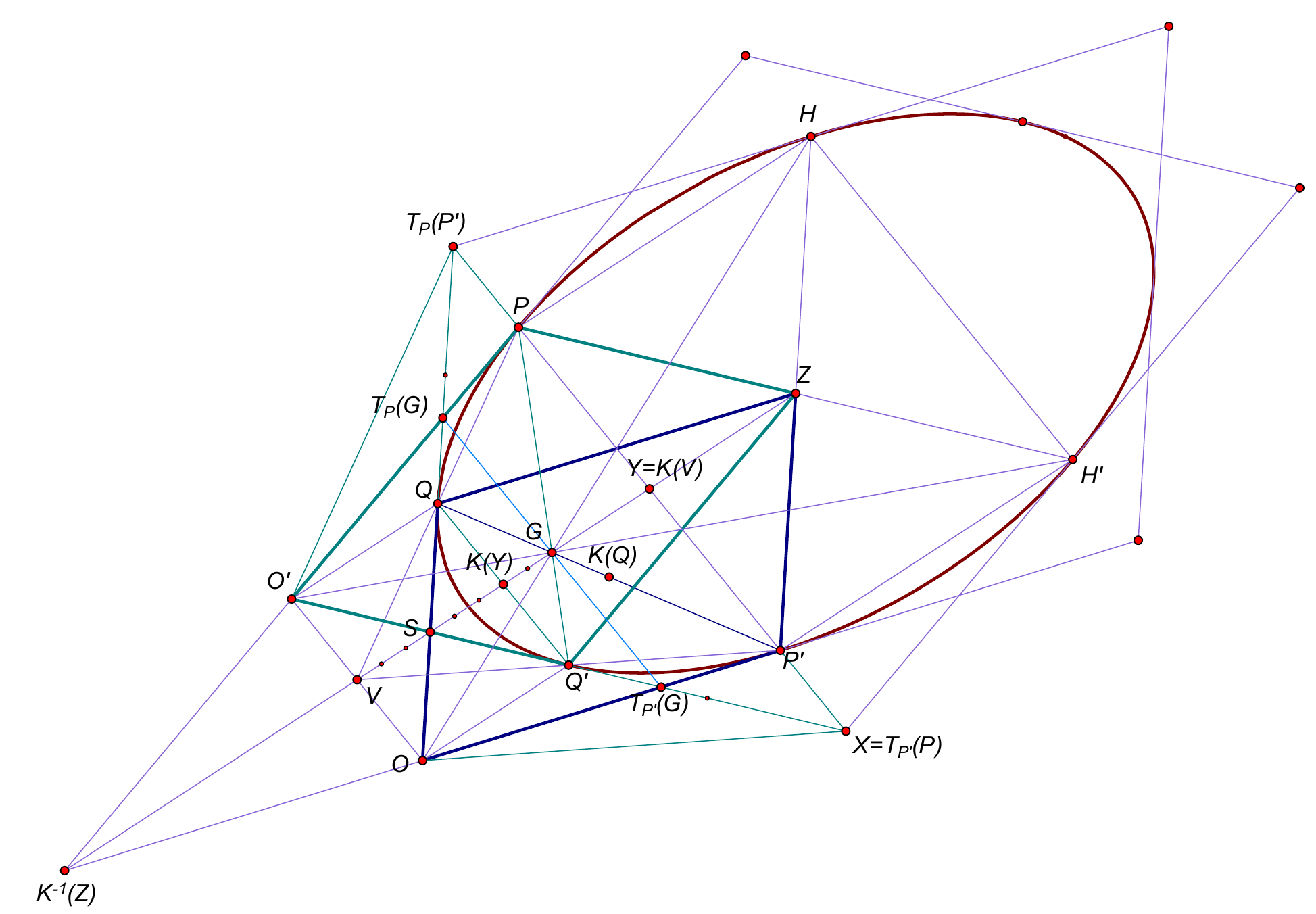}\]
\caption{The parallelograms $QZP'O$ and $Q'ZPO'$ when $\textsf{M}$ is a half-turn.}
\label{fig:3.1}
\end{figure}

\begin {proof} (See Figure \ref{fig:3.1}.)
(1)  As in the proof of Proposition \ref{prop:half-turn} we have $P'=\tilde{Z}=R_O \circ K^{-1}(Z)$, so $K^{-1}(Z)$ lies on $OP'$; by symmetry, it also lies on $O'P$.  We will show that pole of $PP'$ is $K^{-1}(Z)$.  Then (1) follows, since $K^{-1}(Z)$ is conjugate to both $P$ and $P'$, and so lies on the polars of $P$ and $P'$, which are the tangents to $\cC_P$ at $P$ and $P'$.  Hence, $K^{-1}(Z)P=O'P$ and $K^{-1}(Z)P'=OP'$ are tangents to $\cC_P$. \smallskip

We do this by showing that $T_{P'}(P)$ and $T_P(P')$ are conjugate to $K^{-1}(Z)$.  This implies that the polar of $K^{-1}(Z)$ is the join of $T_{P'}(P)$ and $T_P(P')$, which is $PP'$ by II, Corollary 2.2(c).  By symmetry it suffices to consider $T_{P'}(P)$.  We use the fact from Part IV, Prop. 3.10, that $O'Q'$ is tangent to $\cC_P$ at $Q'$.  Applying the map $\lambda$ gives that $\lambda(O'Q')$ is tangent to $\lambda(\cC_P)=\cC_P$ at $\lambda(Q') = H'$, by II, Theorem 3.2 and III, Theorem 2.7.  Using $T_P(P) = O'$ from Proposition \ref{prop:half-turn} we know that
$$\lambda(O') = T_{P'} \circ T_P^{-1}(O') = T_{P'}(P),$$
so $T_{P'}(P)$ lies on the tangent to $\cC_P$ at $H'$ and is conjugate to $H'$.  Also, $P, Q', K^{-1}(P)$ are collinear points, so applying the map $T_{P'}$ gives that
$T_{P'}(P)$ is collinear with $T_{P'}(Q')=Q'$ and
$$T_{P'} \circ K^{-1}(P) = T_{P'} \circ K^{-1} \circ T_P(Q') = \textsf{M}(Q') = O'.$$
Therefore, $T_{P'}(P)$ lies on the tangent $O'Q'$ and so is conjugate to the point $Q'$.  Thus, the polar of $T_{P'}(P)$ is $Q'H'$, which lies on $K^{-1}(Z)$ since $Z, K(Q'), O'=K(H')$ are collinear, using the fact from Corollary \ref{cor:ZS} that $Q'ZPO'$ is a parallelogram and $K(Q')$ is the midpoint of the diagonal $Q'P$.  This shows that $K^{-1}(Z)$ is conjugate to $T_{P'}(P)$, as desired.  Note that this also shows that $T_{P'}(P) \neq T_P(P')$, since the polar of $T_P(P')$ is $QH$, and $QH$ cannot be the same line as $Q'H'$, since the four points $Q,H,Q',H'$ all lie on the conic $\cC_P$. \smallskip

(2) From the fact that $OP'$ and $OQ$ (IV, Prop. 3.10) are tangent to $\cC_P$, it follows that $P'Q$ is the polar of $O$ with respect to $\cC_P$, and likewise, $PQ'$ is the polar of $O'$.  Hence, the pole of $OO'$ is $P'Q \cdot PQ' = G$.  On the other hand, the polar of $G$ is the line $VV_\infty$, where $V_\infty$ is the infinite point on the line $PP'$, by II, Proposition 2.3(a).  Therefore, $V$ lies on $OO'$.  Since $GV$ is the fixed line of the affine reflection $\eta$, $V$ must be the midpoint of segment $OO'$.  Then $V$ lies on the parallel lines $K^{-1}(PP')$ (II, Prop. 2.3(e)) and $OO'$, so $K^{-1}(PP')=OO'$. \smallskip

(3) From (2) we have $PP'=K(OO') = K^2(HH') = NN'$, where $N=K(O)$ and $N'=K(O')$ are the centers of the nine-point conics $\mathcal{N}_H$ and $\mathcal{N}_{H'}$.  Thus, the line $PP'$ is halfway between the parallel lines $OO'$ and $HH'$.  Taking complements, $QQ'$ is halfway between $NN'=PP'$ and $OO'$.  Also, the center $S$ of the map $\textsf{M}$ is located halfway between the lines $QQ'$ and $OO'$, since $\textsf{M}(OO')=QQ'$.  Let $X=T_{P'}(P)=O'Q' \cdot PP'=O'S \cdot PP'$ and $Y=K(V)=GV \cdot PP'$.  Then triangles $VO'S$ and $YXS$ are similar, with similarity ratio $1/3$, because $S$ is the midpoint of segment $O'Q'$ and $Q'$ is the midpoint of segment $O'X$, so that
$$|SX| = |SQ'|+|Q'X|=|SO'|+2|SO'|=3|SO'|.$$
Furthermore, $V$ on $OO'$ implies that $\textsf{M}(V)$ is on $QQ'$, so that $\textsf{M}(V)=QQ' \cdot GV$ is the midpoint of $VY=VK(V)$ and therefore coincides with $K(Y)$.  Then from $|SK(Y)|=|S\textsf{M}(V)|=|SV|$, $|K(Y)G|=\frac{1}{3}|K(Y)Y|$, and $|K(Y)Y|=|VK(Y)|=|V\textsf{M}(V)|=2|SV|$ we find that
\begin{align*}
\frac{|GS|}{|SV|}&=\frac{|SK(Y)|+|K(Y)G|}{|SV|}=\frac{|SV|+|K(Y)Y|/3}{|SV|}\\
&=\frac{|SV|+2|SV|/3}{|SV|}=\frac{5}{3}.
\end{align*}
Since $S$ lies between $G$ and $V$, this proves $\frac{GS}{SV}=\frac{5}{3}$.  Now $\frac{ZG}{GV}=\frac{2GS}{8SV/3}=\frac{3}{4} \frac{GS}{SV}=\frac{5}{4}$.
\end{proof}

\noindent {\bf Remark.} The conditions of Proposition \ref{prop:tangents} are also sufficient for $\textsf{M}$ to be a half-turn.  We leave the verification of this for parts (1) and (2) to the reader.  We will verify this for condition (3) in the next section. 

\end{section}

\begin{section}{Constructing the elliptic curve locus.}

Now suppose a parallelogram $QZP'O$ is given, and with it: $K(Q)$ as the midpoint of $QP'$; $G$ as the point for which the signed distance $QG$ satisfies $QG=\frac{1}{3}QP'$; and $S=K(Z)$ (Corollary \ref{cor:ZS}).  Then Proposition \ref{prop:tangents} shows that the point $V$ is determined by $G$ and $S$.  This determines, in turn, the points $P$ and $Q'$ uniquely, since $P$ is the reflection in $Q$ of the point $V$ (see II, p. 26) and $Q'=K(P)$.  Further, $O'$ is also determined as the reflection of $Q'$ in $S$, or as the reflection of $O$ in $V$.  Therefore, the parallelogram determines $P, Q, P', Q'$ and $H=K^{-1}(O)$, and hence the conic $\cC_P$ on these $5$ points (by III, Theorem 2.8).  Thus, any triangle $ABC$ for which $\textsf{M}$ is a half-turn with the given parallelogram $QZP'O$ must be inscribed in the conic $\cC_P$.  Furthermore, the affine maps $T_P$ and $T_{P'}$ are also determined, since
\begin{equation}
T_P(PQQ')=O'QP \ \ \textrm{and} \ \ T_{P'}(P'QQ')=OP'Q'.
\label{eqn:2.1}
\end{equation}
Defining the maps $T_P$ and $T_{P'}$ by (\ref{eqn:2.1}), we will show that $\textsf{M}=T_P \circ K^{-1} \circ T_{P'}$ is a half-turn about $S$.
\medskip

\begin{lem}
\label{lem:CwithG}
Given collinear and distinct ordinary points $G, V, Z$ and an ordinary point $P$ not on $GZ$, if $Q'=K(P)$ and $P'$ is the reflection of $V$ in the point $Q'$, and $Q$ is the midpoint of $PV$, then:
\begin{enumerate}[(a)]
\item there is a unique conic $\cC$ with center $Z$ which lies on $P, P', Q$, and $Q'$; 
\item with respect to any triangle $ABC$ with vertices on $\cC$ whose centroid is $G$, and whose vertices do not coincide with any of the points $P, P', Q$ or $Q'$, the point $P'$ is the isotomic conjugate of $P$, and $\cC$ coincides with the conic $\cC_P$ for $ABC$.  \end{enumerate}
\end{lem}

\begin{proof}
For (a), first note that $PQ'$, $P'Q'$, and $PQ$ do not lie on $Z$, since $PQ' \cdot GZ = G$ and $P'Q' \cdot GZ = PQ \cdot GZ = V$, which are distinct from $Z$ by assumption.  Suppose that $Z$ is not on $PP'$.  The point $G$, being $2/3$ of the way from $P$ to the midpoint $Q'$ of $VP'$, is the centroid of triangle $PVP'$.  Hence, $K(P')=Q$ and $VG=GZ$ is a median of triangle $PVP'$, implying that $GZ$ intersects $PP'$ at the midpoint of segment $PP'$.  Also, $QQ' \pa PP'$, so $V_\infty=PP' \cdot QQ'$ is on the line at infinity. \smallskip

Now let $\cC$ be the conic with center $Z$, lying on the points $P, Q, P'$.  This exists and is unique, since $Z$ does not lie on $PQ, P'Q$ or $PP'$.  With respect to this conic, $Z$ is conjugate to $V_\infty$, and so is the midpoint $GZ \cdot PP'$, since $P$ and $P'$ lie on $\cC$.  Since $Z$ is not on $PP'$, $GZ$ is the polar of $V_\infty$.  Now, $V_\infty$ lies on $QQ'$, so the point $Q_m = GZ \cdot QQ'$, which is the midpoint of segment $QQ'$, is conjugate to $V_\infty$.  Let $Q^*$ be the second intersection of $QQ'$ with $C$.  Note that $Q^* \neq Q$; otherwise $V_\infty$ would be conjugate to $Q$, so $Q$ would lie on its polar $GZ$, implying that $P$ also lies on $GZ$, which is contrary to assumption.  Hence, $V_\infty$ is conjugate to the midpoint of $QQ^*$, which must be $Q_m$.  This implies that $Q^*=Q'$, so $Q'$ lies on $\cC$. \smallskip

Now suppose $Z$ is on $PP'$. Then $Z$ is not on $QQ'$, since $QQ' \pa PP'$, so there is a unique conic $\cC$ through $P, Q, Q'$ with center $Z$. As above, the pole of $GZ$ is $V_\infty$, and switching the point pairs $P,P'$ and $Q, Q'$ in the argument above gives that $P'$ lies on $\cC$. \smallskip 

For (b), the triangle $ABC$ determines the conic $\cC = \cC_P = ABCPQ'$, since $P$ and $Q'$ cannot lie on any of the sides of $ABC$ and $P$ does not lie on a median of $ABC$ (see the proof of II, Theorem 2.1). We know that this conic has center $Z$, since $ABC$ is inscribed in $\cC$.  Furthermore, the pole of $GZ$ with respect to $\cC$ is $V_\infty = l_\infty \cdot PP'$, as above. But the isotomic conjugate $P^*$ of $P$ with respect to $ABC$ is the unique point $P^* \neq P$ on $PV_\infty$ lying on the conic $\cC_P$ (see II, p. 26), so that means $P^* = P'$.

\end{proof}

\begin{figure}
\[\includegraphics[width=5.5in]{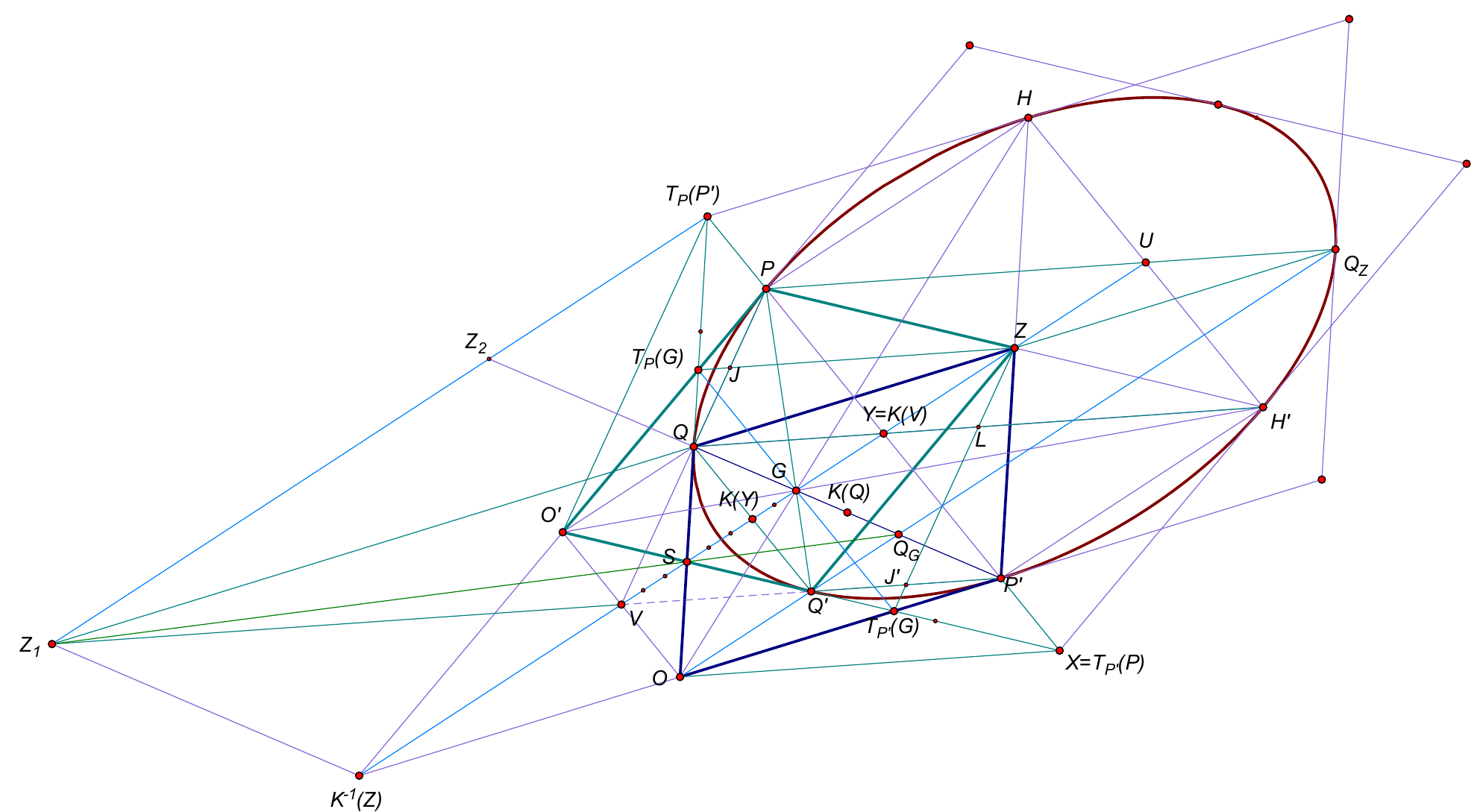}\]
\caption{Proof of Lemma 3.5.}
\label{fig:3.2}
\end{figure}

\begin{lem}
With the assumptions of Lemma \ref{lem:CwithG}, suppose the signed ratio $\frac{ZG}{GV} = \frac{5}{4}$. Then the tangent to the conic $\cC$ at $Q$ is $K(P'Z)=QK(Z)$.
\label{lem:TanParP'Z}
\end{lem}

\begin{proof}
(See Figure \ref{fig:3.2}.)  As in the proof of Proposition \ref{prop:tangents}, in triangle $PP'V$, the midpoint of $PP'$ is the complement $Y=K(V)$ of $V$, and so the midpoint of $QQ'$ is the complement $K(Y)$ of $Y$. Let $U = K^{-1}(V)$. Then $GV = 2 GY = 4GK(Y)$ and $ZG = 5GK(Y)$, so $ZY = ZG - GY = 3GK(Y)$. In addition, $UG = 8GK(Y)$ so $UZ = UG - ZG = 3GK(Y)=ZY$; hence, $Z$ is the midpoint of $YU$. Let $H'$ be the reflection of $P$ in $Z$. Then triangles $H'UZ$ and $PYZ$ are congruent so $H'U \pa PY = P'Y$. Also, triangles $PYZ$ and $PP'H'$ are similar (SAS), so $P'H' \pa YZ=UY$. This implies $P'H'UY$ is a parallelogram and $P'H' = UY = VY$. Hence $P'H'YV$ is also a parallelogram and $H'Y \pa P'V$. Since triangles $QYP$ and $VP'P$ are similar (with similarity ratio $1/2$), $QY \pa P'V$, so $Y$ lies on $H'Q$. \smallskip  

Let $Q_G$ and $Q_Z$ be the reflections of $Q$ in $G$ and $Z$, respectively. Then $QQ_G Q_Z \sim QGZ$, so the line $Q_GQ_Z$ is parallel to $GZ$ and lies halfway between $P'$ and $GZ$.  Hence, $Q'$ lies on $Q_GQ_Z$.  Also, $U$ lies on $PQ_Z$, since $P, U, Q_Z$ are the reflections of $H', Y, Q$ in $Z$, and we proved $H', Y, Q$ are collinear in the previous paragraph. Moreover, $PQ_Z \pa H'Q \pa P'V$, since $P'H'YV$ is a parallelogram. \smallskip

Let $Z_1$ be the reflection of $Q_Z$ in $Q$. Then $Z_1QV \cong Q_ZQP$ (SAS), so $Z_1V \pa PQ_Z \pa P'V$, hence $Z_1$ lies on $P'V$.  Now let $J'$ be the midpoint of $P'Q'$ and $L$ the midpoint of $QH'$.   Then $ZL$ is a midline in triangle $PQH'$, so $ZL \pa PQ = QV$.  Since $P'Q' \pa H'Q = QL$, and $ZL$ lies in the conjugate direction to $QH'$, it follows that $Z, L, J'$  are collinear.  Hence, $Z_1QV \sim Z_1ZJ'$, which implies, since $Z_1Q= 2 QZ$, that $Z_1V=2VJ'= 2 \cdot (\frac{3}{4} VP') = \frac{3}{2}VP'$.  Let $Z^* = K^{-1}(Z)$. Then $Z^*G = 2\cdot GZ = \frac{5}{2} \cdot VG$, by hypothesis, so $Z^*V = \frac{3}{2} \cdot VG$.  Hence, triangles $Z_1Z^*V$ and $P'GV$ are similar (SAS) and $Z_1Z^* \pa P'G$. Let $S = Z_1Q_G \cdot GZ$.  Since $Z_1Z^* \pa Q_GG$, we have similar triangles $Z_1Z^*S$ and $Q_GGS$.  Moreover, $Z_1ZS \sim Z_1Q_ZQ_G$, since $SZ=GZ \pa Q_GQ_Z$, with $Z_1Z=3 ZQ_Z$, so $Z_1S=3SQ_G$.  It follows that $Z^*S=3SG$ and therefore $Z^*G=4SG=2GZ$.  This implies $S = K(Z)$. \smallskip

Let $Z_2$ be the intersection of $P'G$ with the line through $Z_1$ parallel to $GZ$. Also, let $Z_\infty$ be the point at infinity on $GZ$. Then we have the following chain of perspectivities:
\[GVZK(Y) \ \stackrel{Q}{\doublebarwedge} \ P'VZ_1Q' \ \stackrel{Z_\infty}{\doublebarwedge} \ P'GZ_2Q_G \ \stackrel{Z_1}{\doublebarwedge} \ VGZ_\infty S.\]
The resulting projectivity on the line $GZ$ is precisely the involution of conjugate points on $GZ$ with respect to $\cC$, because $G$ and $V$ are conjugate points (they are vertices of the diagonal triangle $GVV_\infty$ of the inscribed quadrangle $PP'QQ'$) and the polar of $Z$ is the line at infinity, which intersects $GZ$ in $Z_\infty$. This gives that $K(Y)$ is conjugate to $S$.  But $S$ is also conjugate to $V_\infty$, since $S$ lies on its polar $GV$.  This implies that the polar of $S$ is $K(Y)V_\infty=QQ'$.  Thus, the tangent to $\cC$ at $Q$ is $QS = K(P'Z)$.
\end{proof}

\begin{prop}
\label{prop:ABCMht}
Under the assumptions of Lemmas \ref{lem:CwithG} and \ref{lem:TanParP'Z}, for any triangle $ABC$ with vertices on $\cC$ whose centroid is $G$, and whose vertices do not coincide with any of the points $P, P', Q$ or $Q'$, the map $\textsf{M}=T_P \circ K^{-1} \circ T_{P'}$ is a half-turn.
\end{prop}
\begin{proof}
By Lemmas \ref{lem:CwithG} and \ref{lem:TanParP'Z}, the tangent at $Q$ to $\cC_P$ goes through $K(Z)$. But the tangent at $Q$ is $OQ$ (IV, Prop. 3.10), hence the generalized insimilicenter $S$ for $ABC$ is $S = OQ \cdot GZ = K(Z)$, where $Z$ is the center of $\cC_P=\cC$. Now the proposition follows from Corollary \ref{cor:ZS}.
\end{proof}

\begin{thm}
\label{thm:GVZ}
Let $G, V, Z$ be any distinct, collinear, and ordinary points with signed ratio $\frac{ZG}{GV} = \frac{5}{4}$, and $P$ an ordinary point not on $GZ$.  Define $Q'=K(P)$ (complement taken with respect to $G$) and let $P'$ be the reflection of $V$ in $Q'$ and $Q$ the midpoint of $PV$.  Finally let $\cC$ be the conic guaranteed by Lemma \ref{lem:CwithG}(a).  For any point $A$ on the arc $\mathscr{A}=PQQ'P'$ of $\cC$ distinct from these four points, there is a unique pair of points $\{B,C\}$ on $\cC$ (and then on the same arc), such that $ABC$ has centroid $G$.  For each such triangle, the map $\textsf{M}$ is a half-turn, and this map is independent of $A$.  Conversely, if $ABC$ is inscribed in the conic $\cC$ with centroid $G$, then $A \in \mathscr{A}-\{P,Q,Q',P'\}$. 
\end{thm}

\begin{proof}
We start by showing that the hypotheses of the theorem can be satisfied for suitable points $G,V,Z,P$, for which the conic $\cC$ is a circle.  Start with a circle $\cC$ with center $Z$.  Pick points $P', Q$ on $\cC$ and $O$ so that $QZP'O$ is a square.  Let $S$ be the midpoint of $OQ$ and $G=SZ \cdot QP'$. Reflect $P'$ and $Q$ in $GZ$ to obtain the points $P$ and $Q'$ on $\cC$.  Let $Y$ on $GZ$ be the midpoint of $PP'$.  Also, let $V =PQ \cdot P'Q'$ on $GZ$.  Since $G=QP' \cdot SZ$ is on the bisector of $\angle SQZ$ and $\frac{ZQ}{QS}=\frac{2}{1}$, we have $\frac{ZG}{GS}=\frac{2}{1}$, so $K(Z)=S$.  This implies that $K(ZP')=SK(P')$ is parallel to $ZP'$ and half the length, so $K(P')=Q$.  In the same way, $K(P)=Q'$.  Then in triangle $PVQ'$ the segment $VG$ bisects the angle at $V$, so $\frac{PV}{VQ'}=\frac{PG}{GQ'}=\frac{P'G}{GQ}=\frac{2}{1}$.  Thus, $VP' = VP = 2 \cdot VQ'$.  It follows that $Q'$ is the midpoint of $VP'$, $Q$ is the midpoint of $VP$, and $G$ is the centroid of $VPP'$.  Hence, $Y=K(V)$.  This shows that the hypotheses of Lemma \ref{lem:CwithG} hold, so $\cC$ is the conic of that lemma.  By the same argument as in the proof of Proposition \ref{prop:tangents}(2), using the fact that $OQ$ and $OP'$ are tangent to $\cC$ at $Q$ and $P'$, respectively, and that $GVV_\infty$ is a self-polar triangle with respect to $\cC$, it follows that $K^{-1}(PP')=OO'$ and $V$ is the midpoint of $OO'$.  Now, letting $\textsf{M}$ be the half-turn about $S$, the same argument as in the proof of Proposition \ref{prop:tangents}(3) gives that $\frac{ZG}{GV}=\frac{5}{4}$.  \smallskip

If $G, V, Z, P$ are any points satisfying the hypotheses, then there is an affine map taking triangle $VPP'$ to the corresponding triangle constructed in the previous paragraph, so that $G$ ( the centroid of $VPP'$) and $Z$ go to the similarly named points and the trapezoid $PP'Q'Q$ is mapped to the corresponding trapezoid for the circle.  Then the image of the new conic $\cC$ is the circle of the previous paragraph, so $\cC$ must be an ellipse.  \smallskip

Given $A$ on the arc $\mathscr{A}=PQQ'P'$ of $\cC$, define $D_0=K(A)$.  Now $P'$ and $K(P')=Q$ are on $\cC$, as are $P$ and $K(P)=Q'$.  We claim that $P$ and $P'$ are the only two points $R$ on $\cC$ for which $K(R)$ is also on $\cC$.  This is because $K(\cC)$ is a conic with center $K(Z)=S$, meeting $\cC$ at $Q, Q'$, and lying on the point $K(Q)$.  Note that the map $K$ fixes all points on $l_\infty$, so $K(\cC)$ induces the same involution on $l_\infty$ that $\cC$ does.  It follows that there are exactly two points in $K(\cC) \cap \cC$.  Since the point $Q$ is on the given arc $\mathscr{A}$ and $K(Q)$, as the midpoint of segment $QP'$, is interior to $\cC$, it follows that the same is true for the point $D_0=K(A)$, for any $A$ on $\mathscr{A}$, while $D_0$ lies outside of $\cC$ when $A$ is on $\cC-\mathscr{A}$.  Now consider the reflection $\cC'$ of the conic $\cC$ in the point $D_0$.  When $D_0$ lies inside $\cC$, it also lies inside $\cC'$, and therefore the two conics $\cC, \cC'$ overlap.  Since reflection in $D_0$ fixes all the points on $l_\infty$, the conic $\cC'$ induces the same involution on $l_\infty$ that $\cC$ does.  Therefore, they have exactly two points in common.  Labeling these points as $B$ and $C$, it is clear that $D_0$ is the midpoint of segment $BC$, and from this and $K(A)=D_0$ it follows that $G$ is the centroid of $ABC$.  On the other hand, if $A$ lies outside of $\mathscr{A}$, then $D_0$ lies outside of $\cC$, and in this case, $\cC'$ does not intersect $\cC$ (so there can be no triangle inscribed in $\cC$ with centroid $G$).  Applying the same argument to the points $B$ and $C$ instead of $A$ shows that $B$ and $C$ are also on the arc $\mathscr{A}$.  The next to last assertion follows from Proposition \ref{prop:ABCMht} and the comments preceding Lemma \ref{lem:CwithG}.
\end{proof}

\begin{figure}
\[\includegraphics[width=4.5in]{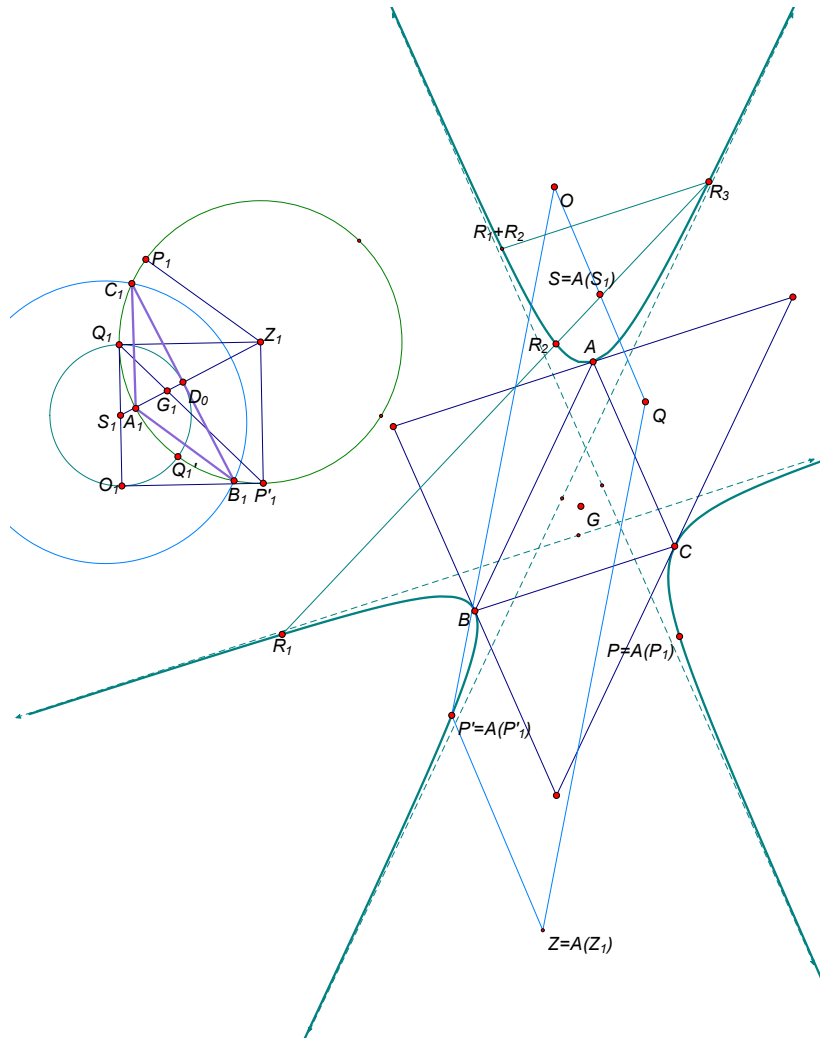}\]
\caption{Elliptic curve locus of $P$ with $\textsf{M}$ a half-turn.}
\label{fig:3.3}
\end{figure}

We can use the proof of Theorem \ref{thm:GVZ} to give a construction of the locus $\mathscr{L}$ of points $P$, for a given triangle $ABC$, for which the map $\textsf{M}$ is a half-turn.  To do this, start with the construction of the points $Q_1Z_1P_1'O_1$ on circle $\cC$, as in the first paragraph of the proof.  Pick a point $A_1$ on the arc $\mathscr{A}=P_1Q_1Q_1'P_1'$, and determine the unique pair of points $B_1, C_1$ on $\mathscr{A}$ so that the centroid of $A_1B_1C_1$ is the point $G_1=Z_1S_1 \cdot Q_1P_1'$, with $S_1$ the midpoint of segment $Q_1O_1$.  Then determine the unique affine map $\textsf{A}$ for which $\textsf{A}(A_1B_1C_1)=ABC$.  The points $P=\textsf{A}(P_1)$ and $P'=\textsf{A}(P_1')$ describe the locus $\mathscr{L}$, as $\textsf{A}$ runs over all affine maps with $A_1 \in \mathscr{A}$.  This locus is shown in Figure \ref{fig:3.3}, and turns out to be an elliptic curve $\mathcal{E}$ minus $6$ points, as we show below.  For the pictured triangle $A_1B_1C_1$ and its half-turn $\textsf{M}_1$, the map $\textsf{M}= \textsf{A} \circ \textsf{M}_1 \circ \textsf{A}^{-1}$ is a half-turn about the point $S=\textsf{A}(S_1)$.  Note that $\mathcal{E}$ is tangent to the sides of the anticomplementary triangle $K^{-1}(ABC)$ of $ABC$ at the vertices.  \medskip

An equation for the curve $\mathcal{E}$ can be found using barycentric coordinates.  It can be shown (see \cite{mm0}, eqs. (8.1) and (3.4)) that homogeneous barycentric coordinates of the points $S$ and $Z$ are
$$S=(x(y+z)^2,y(x+z)^2,z(x+y)^2), \ \ Z=(x(y-z)^2,y(z-x)^2,z(x-y)^2),$$
where $P=(x,y,z)$.  Using the remark after Corollary \ref{cor:ZS}, we compute that the points $P=(x,y,z)$, for which $\textsf{M}$ is a half-turn, satisfy $S=K(Z)$, so the coordinates of $P$ satisfy the equation
$$\mathcal{E}: \ \ x^2(y+z)+y^2(x+z)+z^2(x+y)-2xyz=0.$$
Note that $P \in \mathcal{E} \Rightarrow P' \in \mathcal{E}$.  Setting $z=1-x-y$, where $(x,y,z)$ are absolute barycentric coordinates, we get the affine equation for $\mathcal{E}$:
\begin{equation}
(5x-1)y^2+(5x-1)(x-1)y-x^2+x =0.
\label{eqn:2.2}
\end{equation}
This is the case $a=-5$ of the geometric normal form
$$(ax+1)y^2+(ax+1)(x-1)y+x^2-x=0$$
that we mentioned in \cite{mmv}.  Rational points on $\mathcal{E}$ are $(x,y,z) = (1,0,0)$, $(0,1,0)$, $(0,0,1)$, which correspond to the vertices $A,B,C$.  The points on $l_\infty \cap \mathcal{E}$ are $(x,y,z)=(0,1,-1), (1,0,-1)$, and $(1,-1,0)$, which are the infinite points on the sides of $ABC$.  No other points on the sides or medians of $ABC$ of $K^{-1}(ABC)$ lie on the curve $\mathcal{E}$.  Using (\ref{eqn:2.2}) we can check directly that the curve $\mathcal{E}$ is tangent to $K^{-1}(ABC)$ at the points $A,B,C$ and that it has no singular points.  It follows from this that $\mathcal{E}$ is an elliptic curve, whose points form an abelian group under the addition operation given by the chord-tangent construction.  (See \cite{kna}, p. 67, or \cite{wa}.)  In Figure \ref{fig:3.3} the sum of the points $R_1$ and $R_2$ on $\mathcal{E}$ is the point $R_1+R_2$, taking the point  $A_\infty=BC \cdot l_\infty=(0,1,-1)$ as the base point (identity for the addition operation on the curve). With the base point $A_\infty$, the point $A$ has order $2$, while $B_\infty = AC \cdot l_\infty$ and $C_\infty = AB \cdot l_\infty$ have order $3$, and the points $B,C$ have order $6$.  \medskip

Note that if $P \in \iota(l_\infty)$ is a point on the Steiner circumellipse lying on $\mathcal{E}$, then $P' \in \mathcal{E} \cap l_\infty$, so that $P'$ is one of the points $A_\infty=(0,1,-1), B_\infty=(1,0,-1), C_\infty=(1,-1,0)$, whose isotomic conjugates are $A, B, C$.  Other than the vertices of $ABC$, no points on the Steiner circumellipse lie on $\mathcal{E}$.  Furthermore, the Steiner circumellipse is inscribed in the triangle $K^{-1}(ABC)$, while $\mathcal{E}$ is tangent to the sides of this triangle at $A, B, C$.  By Proposition \ref{prop:ABCMht} and Theorem \ref{thm:GVZ}, any point $P$ for which $\textsf{M}$ is a half-turn has the property that the points $Q=K(P')$ and $Q'=K(P)$ are exterior to triangle $ABC$.  It follows that $P$ and $P'$ are exterior to triangle $K^{-1}(ABC)$.  Hence, {\it all} the points of $\mathcal{E}-\{A,B,C\}$ are exterior to triangle $K^{-1}(ABC)$, as pictured in Figure \ref{fig:3.3}.  \medskip

We now check that the hypotheses of Proposition \ref{prop:half-turn} hold for all the points in $\mathcal{E}-\{A,B,C,A_\infty,B_\infty,C_\infty \}$.  By the results of \cite{mmv}, the points for which the generalized orthocenter $H$ is a vertex are contained in the union of three conics, $\overline{\C}_A \cup \overline{\C}_B \cup \overline{\C}_C$, which lie inside the Steiner circumellipse.  By what we said above, none of the points in $\mathcal{E}-\{A,B,C,A_\infty,B_\infty,C_\infty \}$ can lie on any of these conics, so $H$ is never a vertex for these points.  Hence, the hypotheses of Proposition \ref{prop:half-turn} are satisfied for any $P$ in $\mathcal{E}-\{A,B,C,A_\infty,B_\infty,C_\infty \}$ and Corollary \ref{cor:ZS} implies that the map $\textsf{M}$ for the point $P$ is a half-turn.  Thus, we have the following result.

\begin{thm}
\label{thm:curveE}
The locus of points $P$, not lying on the sides of triangles $ABC$ or $K^{-1}(ABC)$, for which $\textsf{M}=T_P \circ K^{-1} \circ T_{P'}$ is a half-turn, coincides with the set of points whose barycentric coordinates lie in $\mathcal{E}-\{A,B,C,A_\infty,B_\infty,C_\infty \}$.
\end{thm}

Every point $P$ on $\mathcal{E}-\{A,B,C,A_\infty,B_\infty,C_\infty \}$ is a point for which $QZP'O$ is a parallelogram.  This yields an affine map $\textsf{A}$, for which $\textsf{A}(Q_1Z_1P_1'O_1)=QZP'O$, and implies by the proof of Theorem \ref{thm:GVZ} that $\textsf{A}^{-1}(ABC)=A_1B_1C_1$ is a triangle with centroid $G_1$, inscribed on the arc $\mathscr{A}$.  Lemma \ref{lem:CwithG} shows that $P_1$ is the isotomic conjugate of $P_1'$ with respect to $A_1B_1C_1$; hence $P=\textsf{A}(P_1)$.  This shows that every point on $\mathcal{E}$ except the vertices and points at infinity is $P=\textsf{A}(P_1)$ for some affine mapping $\textsf{A}$ in the ``locus" of maps with $A_1 \in \mathscr{A}$.   \medskip

Now, each point $A_1$ on $\mathscr{A}$ yields two points on $\mathcal{E}$, since $\textsf{A}$ maps both points $P_1$ and $P_1'$ to points on $\mathcal{E}$.  Alternatively, with a given triangle $A_1B_1C_1$, there are affine maps $\textsf{A}, \widetilde{\textsf{A}}$ for which $\textsf{A}(A_1B_1C_1)=ABC$ and $\widetilde {\textsf{A}}(A_1C_1B_1)=ABC$.   We claim first that $\textsf{A}(P_1)=-\widetilde{\textsf{A}}(P_1)$, i.e., that $P=\textsf{A}(P_1)$ and $\tilde P=\widetilde{\textsf{A}}(P_1)$ are negatives on the curve $\mathcal{E}$ with respect to the addition on the curve.  This is equivalent to the fact that the line $P \tilde P$ through these two points is parallel to $BC$.   This is obvious from the fact that $\widetilde{\textsf{A}} = \rho \circ \textsf{A}$, where, as in the proof of Lemma \ref{lem:median}, $\rho$ is the affine reflection in the direction of the line $BC$, fixing the points on the median $AD_0=AG$.  \medskip

We claim now that the points on $\mathcal{E}-\{A,B,C,A_\infty, B_\infty, C_\infty\}$ are in $1-1$ correspondence with the collection of point-map pairs $(A_1,\textsf{A})$ and $(A_1,\widetilde{\textsf{A}})$ for $A_1 \in \mathscr{A}-\{P_1,Q_1,P_1',Q_1'\}$.  Suppose that $(A_1,\textsf{A}_1)$ and $(A_2,\textsf{A}_2)$ map to the same point $P$ on $\mathcal{E}$, for $A_1, A_2 \in \mathscr{A}-\{P_1,Q_1,P_1',Q_1'\}$.  Then $\textsf{A}_1(A_1B_1C_1)=ABC=\textsf{A}_2(A_2B_2C_2)$ or $\textsf{A}_1(A_1B_1C_1)=ABC=\textsf{A}_2(A_2C_2B_2)$, since the labeling of the points $B_i, C_i$  can be switched; and for these maps, $\textsf{A}_1(P_1)=P=\textsf{A}_2(P_1)$.  Then $\textsf{A}_1^{-1} \textsf{A}_2(A_2B_2C_2)=A_1B_1C_1$ or $A_1C_1B_1$ and $\textsf{A}_1^{-1} \textsf{A}_2(P_1)=P_1$.  But then the map $\textsf{A}_1^{-1} \textsf{A}_2$ also fixes the points $Q_1,P_1',Q_1'$, since $G_1$ is the centroid for both triangles.  Hence, $\textsf{A}_1^{-1} \textsf{A}_2$ is the identity and $\textsf{A}_1=\textsf{A}_2$, so that $A_1B_1C_1=\textsf{A}_1^{-1}(ABC)=A_2B_2C_2$ or $A_2C_2B_2$.  Therefore, $(A_2,\textsf{A}_2)=(A_1,\textsf{A}_1)$.
 \medskip
 
 Thus, we have proved the following.
 
 \begin{thm} Given a circle $\cC$ with center $Z_1$, points $Q_1$ and $P_1'$ on $\cC$, and point $O_1$ for which $Q_1Z_1P_1'O_1$ is a square, then with the points $G_1, P_1, S_1$ as in Figure \ref{fig:3.3}, the set of points
 $$\mathcal{E}-\{A,B,C,A_\infty, B_\infty, C_\infty\}$$
on the elliptic curve $\mathcal{E}$ coincides with the set of points $\textsf{A}(P_1)$, where $A_1 \in \mathscr{A}=P_1Q_1Q_1'P_1'$ is a point on the arc $\mathscr{A}$ distinct from the points in $\{P_1,Q_1,Q_1', P_1'\}$, $B_1, C_1$ are the unique points on $\mathscr{A}$ for which $A_1B_1C_1$ has centroid $G_1$, and $\textsf{A}$ is an affine map for which $\textsf{A}(A_1B_1C_1)=ABC$ or $\textsf{A}(A_1C_1B_1)=ABC$.
 \end{thm}

Note finally that the discriminant of (\ref{eqn:2.2}) with respect to $y$ is $D=(x-1)(5x-1)(5x^2-2x+1)$, so (\ref{eqn:2.2}) is birationally equivalent to the curve
$$Y^2=(X-1)(5X-1)(5X^2-2X+1).$$
Putting $X=\frac{u}{u-4}, Y=\frac{8v}{(u-4)^2}$ shows that this curve is, in turn, birationally equivalent (over $\mathbb{Q}$) to
\begin{equation}
v^2=(u+1)(u^2+4),
\label{eqn:2.3}
\end{equation}
which has $j$ invariant $j=\frac{2^{4}11^3}{5^2}$.  This curve is curve (20A1) in Cremona's tables \cite{cre}, and has the torsion subgroup $T=\{O,(-1, 0), (0,\pm2), (4,\pm 10)\}$ of order $6$ and rank $r=0$ over $\mathbb{Q}$.  The curve $\mathcal{E}$ has infinitely many real points defined over quadratic extensions of $\mathbb{Q}$, including, for example,
$$P=(-4+\sqrt{19},-1,3), \ \ \left(\frac{9+\sqrt{89}}{2},-2,1\right).$$
This shows that there are infinitely many points for which the map $\textsf{M}$ is a half-turn.   There are even infinitely many such points defined over the field $\mathbb{Q}(\sqrt{6})$, since the points $(u,v)=(2,2\sqrt{6})$ and $(u,v)=(\frac{2}{3},\frac{10\sqrt{6}}{9})$ have infinite order on (\ref{eqn:2.3}).  \medskip

\end{section}

\noindent Dept. of Mathematics, Maloney Hall\\
Boston College\\
140 Commonwealth Ave., Chestnut Hill, Massachusetts, 02467-3806\\
{\it e-mail}: igor.minevich@bc.edu
\bigskip

\noindent Dept. of Mathematical Sciences\\
Indiana University - Purdue University at Indianapolis (IUPUI)\\
402 N. Blackford St., Indianapolis, Indiana, 46202\\
{\it e-mail}: pmorton@math.iupui.edu

\end{document}